\documentclass[reqno,10pt]{amsart}
\usepackage{amsmath,mathrsfs,tikz}
\usepackage{enumerate,hyperref}

\numberwithin{equation}{section}
\newtheorem{theorem}{Theorem}
\newtheorem{lemma}[theorem]{Lemma}
\newtheorem{corollary}[theorem]{Corollary}
\newtheorem{proposition}[theorem]{Proposition}

\theoremstyle{definition}
\newtheorem{definition}[theorem]{Definition}
\newtheorem{example}[theorem]{Example}

\title{$(an+b)$-color compositions}
\author{Daniel Birmajer}
\address{Department of Mathematics\\ Nazareth College\\ 4245 East Ave.\\ Rochester, NY 14618}
\author{Juan B. Gil}
\address{Penn State Altoona\\ 3000 Ivyside Park\\ Altoona, PA 16601}
\author{Michael D. Weiner}

\newcommand{\domino}[2]{%
\tikz[scale=0.5,baseline=4pt]{%
\draw[very thick] (0,0) rectangle (2,1); \draw[thick,gray] (1,0) -- (1,1); 
\node at (0.5,0.5) {$#1$}; \node at (1.5,0.5) {$#2$}
}
}

\begin{document}
\maketitle

\begin{abstract}
For $a,b\in\mathbb{N}_0$, we consider $(an+b)$-color compositions of a positive integer $\nu$ for which each part of size $n$ admits $an+b$ colors. We study these compositions from the enumerative point of view and give a formula for the number of $(an+b)$-color compositions of $\nu$ with $k$ parts. Our formula is obtained in two different ways: 1) by means of algebraic properties of partial Bell polynomials, and 2) through a bijection to a certain family of weak compositions that we call {\em domino compositions}. We also discuss two cases when $b$ is negative and give corresponding combinatorial interpretations.
\end{abstract}

\section{Introduction}
\label{sec:introduction}

A composition of a positive integer $\nu$ with $k$ parts is an ordered $k$-tuple $(j_1, \dotsc, j_k)$ of positive integers called parts such that $j_1 + \dotsb + j_k = \nu$.

Given a sequence of nonnegative integers $w=(w_n)_{n\in\mathbb{N}}$, we define a $w$-color composition of $\nu$ to be a composition of $\nu$ such that part $n$ can take on $w_n$ colors. If $w_{n} = 0$, it means that we do not use the integer $n$ in the composition. Such colored compositions have been considered by many authors and continue to be of current interest. For a comprehensive account on the subject, we refer to the book by S.~Heubach and T.~Mansour \cite{HM10}.

If we let $W_n$ be the number of $w$-color compositions of $n$, Moser and Whitney \cite{MW61} observed that the generating functions $w(t)=\sum_{n=1}^\infty w_n t^n$ and $W(t)=\sum_{n=1}^\infty W_n t^n$ satisfy the relation $W(t) = \frac{w(t)}{1-w(t)}$, which means that the sequence $(W_n)_{n\in\mathbb{N}}$ is the invert transform of $(w_n)_{n\in\mathbb{N}}$.

In this paper, we consider the sequence of colors $w_n = an + b$ for $n \ge 1$, with $a, b \in \mathbb{N}_0$.  Thus $w(t) = \sum_{n=1}^\infty (an+b) t^n$, and we have
\begin{equation*}
  w(t) = \sum_{n=1}^\infty (an+b) t^n = \dfrac{at}{(1-t)^2}+  \dfrac{bt}{1-t} = \dfrac{(a+b)t-bt^2}{(1-t)^2}.
\end{equation*}
Therefore, $W(t) = \dfrac{w(t)}{1- w(t)} = \dfrac{(a+b)t-bt^2}{1 - (a+b+2)t + (b+1)t^2}$, and so the number $W_\nu$ of $(an + b)$-color compositions of $\nu$ satisfies the recurrence relation
\begin{equation}\label{eq:recurrence}
  W_{\nu} = (a + b + 2) W_{\nu-1} - (b + 1) W_{\nu-2} \;\text{ for } \nu > 2,
\end{equation}
with initial conditions $W_1 = a+ b$ and $W_2 = (a + b)^2 + (2a + b)$. 

\section{Colored compositions with $k$ parts}
\label{sec:comp-with-parts}

Let $c_{n,k}(w)$ be the number of $w$-color compositions of $n$ with exactly $k$ parts. In \cite{HL68}, Hoggatt and Lind derived the formula
\begin{equation} \label{eq:HL}
  c_{n,k}(w)=\sum_{\pi_k(n)} \frac{k!}{k_1!\cdots k_n!}\, w_1^{k_1}\cdots w_n^{k_n},
\end{equation}
where the sum runs over all $k$-part partitions of $n$, i.e.\ over all solutions of 
\[ k_1+2k_2+\cdots+nk_n=n \text{ such that } k_1+\cdots+k_n=k \]
with $k_j\in\mathbb{N}_0$ for all $j$. Observe that the right-hand side of \eqref{eq:HL} is precisely the $(n,k)$-th partial Bell polynomial $B_{n,k}(1!w_1, 2!w_2,\dots)$ multiplied by the factor $k!/n!$. Thus \eqref{eq:HL} may be written as
\begin{equation} \label{eq:Bell}
  c_{n,k}(w) = \frac{k!}{n!} B_{n, k} (1!w_1, 2! w_2, \dots),
\end{equation}
and the total number of such compositions of $n$ is $W_n = \sum_{k=1}^n c_{n,k}(w)$.

\begin{proposition}
Let $x=(x_n)_{n\in\mathbb{N}}$ and $y=(y_n)_{n\in\mathbb{N}}$ be sequences of nonnegative integers, and let $a,b\in \mathbb{Z}$. Letting $c_{0,0}(w)=1$ and $c_{m,j}(w)=0$ for $m<j$, we have
\[ c_{n,k}(ax+by) =  \sum_{m=0}^n\sum_{j = 0}^k \binom{k}{j} a^{j}b^{k-j}c_{m, j}(x) c_{n-m, k-j}(y). \]
\end{proposition}
\begin{proof}
Since $c_{n,k}(w) = \frac{k!}{n!} B_{n, k} (1!w_1,2!w_2,\dots)$, we can use basic properties of the partial Bell polynomials (see e.g.\ \cite[Sec.~3.3]{Comtet}) together with the notation $!w=(n!w_n)$ to get 
\begin{align*}
  c_{n,k}(ax+by) &= \frac{k!}{n!} B_{n, k} (!(ax+by)) \\
  &= \frac{k!}{n!} \sum_{m=0}^n\sum_{j = 0}^k \binom{n}{m} B_{m, j}(!(ax)) B_{n-m, k-j} (!(by)) \\
  &= \frac{k!}{n!} \sum_{m=0}^n\sum_{j = 0}^k \binom{n}{m} a^j B_{m, j}(!x) b^{k-j}B_{n-m, k-j} (!y) \\
  &= \frac{k!}{n!} \sum_{m=0}^n\sum_{j = 0}^k \binom{n}{m} a^{j}b^{k-j}
  \frac{m!}{j!} c_{m, j}(x) \frac{(n-m)!}{(k-j)!} c_{n-m, k-j} (y) \\
  &= \sum_{m=0}^n\sum_{j = 0}^k \binom{k}{j} a^{j}b^{k-j} c_{m, j}(x) c_{n-m, k-j}(y).
\end{align*}
\end{proof}

\begin{theorem}
The number of $(an+b)$-color compositions of $\nu$ with $k$ parts is given by
\[ c_{\nu,k}(an+b) = \sum_{j=0}^{k} a^j b^{k-j} \binom{k}{j} \binom{\nu+j-1}{\nu-k}. \]
Thus the total number of $(an+b)$-color compositions of $\nu$ is 
\begin{equation}\label{eq:total_compositions}
 W_\nu = \sum_{k=1}^{\nu} \sum_{j=0}^{k} a^j b^{k-j} \binom{k}{j} \binom{\nu+j-1}{\nu-k}.
\end{equation}
\end{theorem}

\begin{proof}
We use the above proposition with the sequences $x_n=n$ and $y_n=1$. Then
\begin{align*}
c_{\nu,k}(an+b) &=  \sum_{m=0}^\nu\sum_{j = 0}^k \binom{k}{j} a^{j}b^{k-j}c_{m, j}(n) c_{\nu-m, k-j}(1) \\
&= \sum_{j = 0}^k a^{j}b^{k-j}\binom{k}{j} \sum_{m=j}^{\nu} \binom{m+j-1}{m-j} \binom{\nu-m-1}{k-j-1} \\
&= \sum_{j = 0}^k a^{j}b^{k-j}\binom{k}{j} \sum_{\ell=0}^{\nu-j} \binom{\ell+2j-1}{\ell} \binom{\nu-\ell-j-1}{k-j-1} \\
&= \sum_{j = 0}^k a^{j}b^{k-j}\binom{k}{j} \sum_{\ell=0}^{\nu-k} \binom{\ell+2j-1}{\ell} \binom{\nu-\ell-j-1}{\nu-k-\ell} \\
&= \sum_{j = 0}^k a^{j}b^{k-j}\binom{k}{j} (-1)^{\nu-k} \sum_{\ell=0}^{\nu-k}\binom{-2j}{\ell} \binom{j-k}{\nu-k-\ell} \\
&= \sum_{j = 0}^k a^{j}b^{k-j}\binom{k}{j} (-1)^{\nu-k} \binom{-j-k}{\nu-k} \\
&= \sum_{j = 0}^k a^{j}b^{k-j}\binom{k}{j} \binom{\nu+j-1}{\nu-k}.
\end{align*}
\end{proof}

\begin{example}
For some values of $a$ and $b$, \eqref{eq:total_compositions} gives nice formulas for the following sequences, listed in the OEIS \cite{Sloane}:

\medskip
\begin{center}
\begin{tabular}{r|ccr|c}
Compositions & Sequence &\hspace{2em} & Compositions & Sequence \\ \hline
\rule[-1ex]{0ex}{3.8ex} 
$n$-color & A001906 && $(2n-1)$-color & A003946\\[2pt]
$(n+1)$-color & A003480 && $2n$-color & A052530\\[2pt]
$(n+2)$-color & A010903 && $(2n+1)$-color & A060801\\[2pt]
$(n+3)$-color & A010908 && $(3n-1)$-color & A055841\\[2pt]
\end{tabular}
\end{center}
\end{example}

\section{Combinatorial interpretation: Domino compositions}
\label{sec:combinatorial}

In this section, an {\em $n$-domino} is a tile of the form
\begin{equation}\label{def:domino}
\domino{\alpha}{\beta} \;\text{ with $0<\alpha\le n$ and $0\le \beta\le n$}. \smallskip
\end{equation}

A domino with $\beta=0$ will be called a {\em zero $n$-domino}. An $n$-domino composition of $\nu$ is a weak composition of $\nu$ using $n$-dominos of the form \eqref{def:domino}. For example,
\begin{equation*}  
  \domino{1}{1}\, \domino{4}{0}\, \domino{1}{3}
\end{equation*}
is a domino composition of 10 with 4-dominos corresponding to the weak composition $(1,1,4,0,1,3)$.

\begin{definition}
For $a,b\in\mathbb{N}_0$, $n,k\in\mathbb{N}$, and $j\le k$, let $T^{a,b}_j(n,k)$ be the set of $n$-domino compositions of $n+j$ with $j$ nonzero $n$-dominos, available in $a$ different colors, and $k-j$ zero $n$-dominos available in $b$ different colors. Let $T^{a,b}(n,k) = \bigcup_{j=0}^n T^{a,b}_j(n,k)$.
\end{definition}

\begin{lemma} \label{lem:j_nonzero}
\[ \left|T^{a,b}_j(n,k)\right| = a^jb^{k-j} \binom{k}{j}\binom{n+j-1}{n-k}. \]
\end{lemma}
\begin{proof}
Having $k$ dominos, there are $\binom{k}{j}$ ways to choose the $j$ nonzero dominos. Once the dominos are chosen, there are $2j+(k-j)=k+j$ spaces to place positive numbers whose sum is $n+j$. These are compositions of $n+j$ with $k+j$ parts and there are $\binom{n+j-1}{k+j-1}=\binom{n+j-1}{n-k}$ of them. Since the nonzero dominos come in $a$ colors and the zero dominos in $b$ colors, we need to multiply by $a^jb^{k-j}$ to account for all of the possibilities.
\end{proof}

\begin{theorem} \label{thm:bijection}
For any given $a,b\in \mathbb{N}_0$, there is a bijection $\varphi$ between $T^{a,b}(\nu,k)$ and the set of $(an+b)$-color compositions of $\nu$ with $k$ parts. 
\end{theorem}

\begin{proof}
We start by discussing the case when $a=1$. Let $(D_{1},\dots,D_{k})$ be an element of $T^{1,b}(\nu,k)$ with $j$ nonzero $\nu$-dominos. For a nonzero domino $D$, we define $\varphi(D)$ by 
\begin{equation*}
 \domino{\alpha}{\beta} \longrightarrow (\alpha+\beta-1)_{\beta},
\end{equation*}
where the notation $(i)_\ell$ means part $i$ with color $\ell$. For a zero domino $D$ with color $\delta\le b$, we define $\varphi(D)$ by
\begin{equation*}
 \domino{\alpha}{0}_{\,\delta} \longrightarrow (\alpha)_{\ell}, \text{ where } \ell = \alpha + \delta.
\end{equation*}
If we denote the nonzero dominos by $(\alpha_1,\beta_1)$, \dots, $(\alpha_j,\beta_j)$, and the zero dominos by $(\alpha_{j+1},0)$, \dots, $(\alpha_k,0)$, then by definition $(\alpha_1+\beta_1)+\dots+(\alpha_j+\beta_j)+\alpha_{j+1}+\dots+\alpha_k = \nu+j$, and therefore, $(\alpha_1+\beta_1-1)+\dots+(\alpha_j+\beta_j-1)+\alpha_{j+1}+\dots+\alpha_k = \nu$. In other words, $(\varphi(D_{1}),\dots,\varphi(D_{k}))$ is an $(n+b)$-color composition of $\nu$ with $k$ parts.

Conversely, let $\big((i_1)_{\ell_1},\dots,(i_k)_{\ell_k}\big)$ be an $(n+b)$-color composition of $\nu$ such that $j$ of its parts are of the form $(i)_\ell$ with $\ell\le i$.  If $(i)_\ell$ is such a part, then we define $\psi((i)_\ell)$ by
\begin{equation*}
 (i)_\ell \longrightarrow \domino{\alpha_i}{\ell}\,, \text{ where } \alpha_i = i-\ell+1,
\end{equation*}
and if part $(i)_\ell$ is such that $\ell=i+\delta>i$, then we define $\psi((i)_\ell)$ by  
\begin{equation*}
 (i)_\ell \longrightarrow \domino{i}{0}_{\,\delta}.
\end{equation*}
In particular, the components of a domino $\psi((i)_\ell)$ add to $i+1$ if $\ell\le i$ or they add to $i$ if $\ell>i$. Since $i_1+\dots+i_k=\nu$, we get that $\big(\psi((i_1)_{\ell_1}),\dots,\psi((i_k)_{\ell_k})\big)$ is a $\nu$-domino composition in $T^{1,b}_j(\nu,k)$. Clearly, $\psi$ is the inverse of $\varphi$.

\medskip 
For $a>1$ the argument is similar. In this case, for a nonzero domino $D_\gamma\in T^{a,b}(\nu,k)$ with color $1\le \gamma \le a$, we define $\varphi(D_\gamma)$ by
\begin{equation*}
 \domino{\alpha}{\beta}_{\,\gamma} \longrightarrow (\alpha+\beta-1)_{\ell}, \text{ where } \ell = (\alpha+\beta-1)(\gamma-1)+\beta,
\end{equation*}
and for a zero domino $D_\delta$ with color $1\le\delta\le b$, we define $\varphi(D_\delta)$ by
\begin{equation*}
 \domino{\alpha}{0}_{\,\delta} \longrightarrow (\alpha)_{\ell}, \text{ where } \ell = a\alpha + \delta.
\end{equation*}

The inverse map is obtained as follows. For a part $i$ with color $\ell$, $1\le \ell\le ai+b$, write $\ell=qi+r$ with $0<r\le i$ and define a $\nu$-domino as follows:
\begin{align*}
 \text{ if } q<a &: \;\; (i)_\ell \longrightarrow \domino{\alpha_i}{r}_{\,q+1} \text{ with } \alpha_i = i-r+1, \\
 \text{ if } q=a &: \;\; (i)_\ell \longrightarrow \domino{i}{0}_{\,r}, 
\end{align*}
where the subscript outside the domino indicates its color.
\end{proof}

\begin{example}
In the context of $(n+2)$-color compositions, we have

\begin{center}
\small
\begin{tabular}{lll}
 $1_1$ & $\leftrightarrow$ & $\domino{1}{1}$ \\[5pt]
 $1_2$ & $\leftrightarrow$ & $\domino{1}{0}_{\,1}$ \\[5pt]
 $1_3$ & $\leftrightarrow$ & $\domino{1}{0}_{\,2}$
\end{tabular} 
\begin{tabular}{lll}
 $2_1$ & $\leftrightarrow$ & $\domino{2}{1}$ \\[5pt]
 $2_2$ & $\leftrightarrow$ & $\domino{1}{2}$ \\[5pt]
 $2_3$ & $\leftrightarrow$ & $\domino{2}{0}_{\,1}$ \\[5pt]
 $2_4$ & $\leftrightarrow$ & $\domino{2}{0}_{\,2}$
\end{tabular} 
\begin{tabular}{lll}
 $3_1$ & $\leftrightarrow$ & $\domino{3}{1}$ \\[5pt]
 $3_2$ & $\leftrightarrow$ & $\domino{2}{2}$ \\[5pt]
 $3_3$ & $\leftrightarrow$ & $\domino{1}{3}$ \\[5pt]
 $3_4$ & $\leftrightarrow$ & $\domino{3}{0}_{\,1}$ \\[5pt]
 $3_5$ & $\leftrightarrow$ & $\domino{3}{0}_{\,2}$
\end{tabular}
\medskip
\end{center}
For example, the composition $(3_5,1_2,3_2)$ of $7$ corresponds to 
\[ \domino{3}{0}_{\,2}\; \domino{1}{0}_{\,1}\; \domino{2}{2}. \] 
\end{example}

\medskip
As a direct consequence of Theorem~\ref{thm:bijection} and Lemma~\ref{lem:j_nonzero}, we obtain:
\begin{corollary}
The number of $(an+b)$-color compositions of $\nu$ with $k$ parts is given by
\begin{equation*}
 c_{\nu,k}(an+b) = \sum_{j=0}^k a^jb^{k-j} \binom{k}{j}\binom{\nu+j-1}{\nu-k}.
\end{equation*}
\end{corollary}

\section{Other examples}
\label{sec:other}

We finish with two examples related to $(n-1)$-color and $(n-2)$-color compositions.

\begin{example}[$a=1$, $b=-1$]
In this case, we have that
\[ c_{\nu,k}(n-1) = \sum_{j=0}^{k} (-1)^{k-j} \binom{k}{j} \binom{\nu+j-1}{\nu-k} \]
is the number of compositions of $\nu$ with $k$ parts with no part 1 and such that each part $i>1$ may be colored in $i-1$ different ways. This is also the number of $n$-color compositions of $\nu$ with $k$ parts and no color 1.
\end{example}

\begin{example}[$a=1$, $b=-2$]
Let $\mathscr{C}_{n-2}(\nu,k)$ be the set of compositions of $\nu$ with $k$ parts such that:
\begin{itemize}
\item[$\circ$] there is no part 2
\item[$\circ$] each part $i>2$ maybe colored in $i-2$ different ways.
\end{itemize}
\smallskip
If $\mathscr{C}^{1,\text{even}}_{n-2}(\nu,k)$ denotes the set of compositions in $\mathscr{C}_{n-2}(\nu,k)$ with an even number of 1's, and $\mathscr{C}^{1,\text{odd}}_{n-2}(\nu,k)$ is the set of compositions with an odd number of 1's, then we have
\begin{equation*}
 c_{\nu,k}(n-2) = \big|\mathscr{C}^{1,\text{even}}_{n-2}(\nu,k)\big| - \big|\mathscr{C}^{1,\text{odd}}_{n-2}(\nu,k)\big|,
\end{equation*}
which implies
\begin{gather*}
 c_{\nu,\nu}(n-2) = (-1)^\nu, \\
 c_{\nu,k}(n-2) = \sum_{j=1}^{k} (-1)^{k-j} \binom{k}{j} \binom{\nu-k-1}{2j-1} \;\text{ for } k<\nu.
\end{gather*}

\medskip
Moreover, by \eqref{eq:recurrence}, the sequence defined by $W_\nu = \sum_{k=1}^{\nu} c_{\nu,k}(n-2)$ satisfies the recurrence relation
\begin{gather*}
 W_1 = -1, \quad W_2=1, \\
  W_{\nu} = W_{\nu-1} + W_{\nu-2} \;\text{ for } \nu >2.
\end{gather*}
In other words, $W_\nu$ is the Fibonacci number $F_{\nu-3}$ and we get the identity
\[  F_{\nu-3} = \sum_{k=1}^{\nu} \sum_{j=1}^{k} (-1)^{k-j} \binom{k}{j} \binom{\nu-k-1}{2j-1}. \]
\end{example}

\subsection*{Acknowledgement}
We are grateful for the opportunity to present at the ``48th Southeastern International Conference on Combinatorics, Graph Theory \& Computing'' in the spring of 2017. The results of this paper were inspired by Brian Hopkins' talk on {\em Color Restricted n-Color Compositions} and further conversations with him during the conference. 


\end{document}